\documentclass[11pt,a4paper]{article}

\input{headers.tex}

\begin{document}

\title{Necessary and sufficient condition for the existence of a Fréchet mean on the circle}

\author{Benjamin Charlier\\
Institut de Math\'ematiques de Toulouse\\
Universit\'e de Toulouse et CNRS (UMR 5219) \\
{\small {\tt Benjamin.Charlier@math.univ-toulouse.fr} }}

\maketitle

\begin{abstract}
Let $(\S^1,d_{\S^1})$ be the unit circle in $\R^2$ endowed with the arclength distance. We give a sufficient and necessary condition for a general probability measure $\mu$ to admit a well defined Fréchet mean on $(\S^1,d_{\S^1})$. 
We derive a new sufficient condition of existence $P(\alpha,\varphi)$ with no restriction on the support of the measure. Then, we study the convergence of the empirical Fréchet mean to the Fréchet mean and we give an algorithm to compute it.
\end{abstract}
\begin{center}
\end{center}

\section{Introduction}

\subsection{Statistics for non-Euclidean data.} In many fields of interest, results of an experiment are objects taking values in non-Euclidean spaces.  A rather general framework to model such data is Riemannian geometry and more particularly quotient manifolds. As an illustration, in biology or geology, directional data are often used, see e.g. \cite{MR1828667} or  \cite{MR1251957} and references therein. In this case, observations take their values in the circle or a sphere, that is, an Euclidean space quotiented by the action of scaling. 

The usual definitions of basic statistical concepts were developed in an Euclidean framework. Therefore, these definitions must be adapted for random variables with values in non-Euclidean spaces such as manifolds. To describe the localization of a probability distribution, one needs to define a central value such as a mean or a median. 
There has been multiple attempts to give a  definition of a mean in non Euclidean space, see among many others \cite{Bhatt,MR1705458,MR1842295,MR2085875,MR0442975,MR1187782,MR1370296} or \cite{MR0027464}.

In this paper, we consider the so-called Fréchet mean, see \cite{MR0027464,MR0442975,MR1842295} or \cite{Bhatt} and references therein. We are particularly interested in the study of its uniqueness. The Fréchet mean is defined on general metric spaces by extending the fact that Euclidean mean minimizes the sum of square the distance to the data, see equation \eqref{eq.Fre} below. To study the well definiteness of the Fréchet mean on a manifold, two facts must be taken into account: non uniqueness of geodesics from one point to another  (existence of a cut locus) and the effect of curvature, see e.g. \cite{MR1705458} for further discussion. Due to the cut-locus, the distance function is no longer convex and finding conditions to ensure the uniqueness of the Fréchet mean is not obvious. Two main directions have been explored  in the literature: bounding the support of the measure  in \cite{Buss:2001:SAA:502122.502124} for the $n$-spheres and in \cite{MR0442975,MR1891212,MR1842295,MR2736346} for manifolds,  or consider special cases of  absolutely continuous radial distributions, see  \cite{Kaziska20081314} for the unit circle and  or \cite{MR1891212,MR1618880} for projective spaces. In a sense, these two conditions control the concentration of the probability measure. The philosophy behind these works is  to ensure a convexity property of the Fréchet functional given by equation \eqref{eq.F} below, see e.g. the introduction of \cite{MR2736346} for a review of the above cited papers.

\subsection{Fréchet mean on the circle}

A standard way to extend the definition of the Euclidean mean in non-Euclidean metric space is to use the minimization property of the Euclidean mean. This definition is usually credited to M. Fréchet in \cite{MR0027464} although some authors credit it to E. Cartan, see e.g. \cite{MR1891212}. Let $\S^1$ be the unit circle  of the plane,
$$
\S^1 = \{x_1^2+x_2^2 =1,\ (x_1,x_2)\in\R^2 \}.
$$
endowed with the arclength distance given for all $x=(x_1,x_2),p=(p_1,p_2)\in\S^1$ by
\begin{equation}\label{eq:dist0}
d_{\S^1}(x,p) = 2 \arcsin\left(\frac{\norm{x- p}}{2}\right),
\end{equation}
where $\norm{x - p} = \sqrt{(x_1-p_1)^2 + (x_2 - p_2 )^2}$ is the Euclidean norm in $\R^2$. In the whole paper, $\mu$ is a probability measure on $\S^1$ and the Fréchet functional is defined for all $p\in \S^1$ by 
\begin{equation}\label{eq.F}
F_\mu (p) = \frac{1}{2} \int_{\S^1} d^2(x,p) d\mu(x).
\end{equation}
The Fréchet functional is Lipschitz since by the triangle inequality we have 
$
|F_\mu(p_1)-F_\mu(p_2)| \leq 2\pi d_{\S^1}(p_1,p_2)
$ for any $p_1,p_2\in\S^1$. Thus,  $F_{\mu}$  attains its minimum in at least one point and the only issue at hand is uniqueness.
\begin{definition}
	We say that the Fréchet mean of a probability measure $\mu$ in $(\S^1,d_{\S^1})$ is well defined if $F_\mu$ admits a \emph{unique} argmin. That is, there exists a unique $p^*\in \S^1$ satisfying $F_\mu (p^*) = \min_{p\in M} F_{\mu}(p)$, and we note 
\begin{equation}\label{eq.Fre}
p^* = \argmin_{p\in \S^1} F_{\mu}(p).
\end{equation} 
\end{definition}
The argmins of $F_\mu$  are also called Riemannian center of mass \cite{MR1370296} or intrinsic mean \cite{Bhatt} 
as  the  $(\S^1,d_{\S^1})$ is a simple one dimensional compact Riemannian manifold. The advantage of dealing with a simple object such as the circle is that curvature problems disappear and we only face the cut-locus problem. In this sense, it allows us to completely understand its effect on the non-convexity of the distance function $d_{\S^1}$, and to give a complete answer about the problem of uniqueness. In what follows, we fully characterize probability measures that admit a well defined Fréchet mean on the circle $(\S^1,d_{\S^1})$. In particularly  a necessary and sufficient condition is given in Theorem \ref{th:crit}, which  links  the existence of a Fréchet mean for a measure $\mu$ to the comparison between the distribution $\mu$ and the uniform measure $\lambda$ on $\S^1$. 
The surprising fact is that $\lambda$ appears as a benchmark to discriminate measures  having a well defined Fréchet mean.  The uniform measure $\lambda$ is the 'worst' possible case  as all points of the circle is a Fréchet mean, indeed the Fréchet functional \eqref{eq.F} is constant and equals to $\frac{\pi^3}{3}$.

In opposition to what have been done before we do not try to ensure convexity property on the Fréchet functional. Indeed,  the definition of the Fréchet mean relies on the \emph{global} optimization problem \eqref{eq.Fre}  which is, in general, non convex. The advantage of our approach is that we do not need to restrict the support or suppose restrictive conditions of symmetry on the density. As the geometry of flat manifold is simple, we can derive explicit form on the Fréchet functional and its derivative which can be hard to compute  in non-flat manifolds such as $n$-dimensional spheres. 

\subsection{Organization of the paper}

In Section \ref{part.notation},  we introduce notations that will be used throughout the paper. In Section \ref{part.ff},  we give explicit expressions for the Fréchet functional and its derivative and we discuss some properties of critical points of the Fréchet functional. Section \ref{part.th} contains the main result with the necessary and sufficient condition of Theorem \ref{th:crit} for the existence of the Fréchet mean for a general measure. We also 
propose a new sufficient criterion $P(\alpha,\varphi)$ that ensures the well definiteness of the Fréchet mean. In Section \ref{part.emp}, we study the convergence of the empirical Fréchet mean to the Fréchet mean, and describe an algorithm to compute the empirical Fréchet mean.

\section{Notations}
 \label{part.notation}

In what follows,  $\1_{A}$ denotes the indicator function of the set $A\subset \R$  and the notation $\int_{a}^{b}f(t)d\mu_{p_0}(t) $ stands for the Lebesgue integral $\int_{[a,b[}f(t)d\mu_{p_0}(t)$ if $a\leq b$ and $\int_{[b,a[}f(t)d\mu_{p_0}(t) $ if $b>a $.

\subsection{Normal coordinates}

Given a base point $p\in\S^1$, there is a canonical chart called  the exponential map defined from $T_p\S^1 \simeq \R$, the tangent space of $\S^1$ at $p$, to $\S^1$ and denoted by 
\begin{equation*}
	\begin{array}{rcl}
		e_p : \R & \longrightarrow &\S^1 \\
		\theta & \longmapsto & e_p(\theta) = R_{\theta} p,\end{array} \qquad \text{ where } R_{\theta} = \begin{pmatrix}
		\cos(\theta) & -\sin(\theta) \\\sin(\theta) & \cos(\theta)\\ 
	\end{pmatrix} 
\end{equation*}
This map is onto but not one to one as it is $2\pi$-periodic. To guaranty the injectivity, we choose to restrict the domain of definition $\R$ of $e_p$ to $[-\pi,\pi[$. Thus, for all $p_1,p_2\in\S^1$ there is a now unique $\theta_{p_2}^{p_1} \in[-\pi,\pi[$ satisfying $e_{p_1}(\theta_{p_2}^{p_1}) = p_2$ and 
$$
e_p :[-\pi,\pi[\longrightarrow \S^1 \qquad \text{ and } \qquad    e_p^{-1} :  \S^1 \longrightarrow [-\pi,\pi[, \qquad \text{ for all }p\in\S^1.
$$
Such parametrizations are called normal coordinates systems centered at $p$ and $\theta_{p_2}^{p_1}$ is nothing else but the coordinate of $p_2$ read in  a normal  coordinate system centered in $p_1$. To simplify the notations,  we will omit the exponent $p_1$ if no confusion is possible and we will write $\theta_{p_2}^{p_1} = \theta_{p_2}$.  


The cut locus of a point $p_0\in\S^1$ is denoted by $\tilde{p_0}$ and is equal to the opposite point (in $\R^2$) of $p_0$, that is $\tilde p_0 = -p_0$. In a normal coordinates system centered at $p\in\S^1$, the coordinate of  $\tilde p_0$ is $\theta_{\tilde p_0}^p = \theta_{p_0}^p - \pi$ if $0\leq \theta_{p_0}^p<\pi$ or $ \theta_{\tilde p_0}^p = \theta_{p_0}^p + \pi$ if $\pi\leq \theta_{p_0}^p<0$.

\subsection{Distance function and probability measures and the Fr\'echet functional}

The arclength distance between two points $p_1,p_2\in\S^1$ was defined in \eqref{eq:dist0}. Given normal coordinates $\theta_{p_1},\theta_{p_2}\in [-\pi,\pi[$ of these points in the chart centered at an arbitrary point , 
$$
d_{\S_1}(p_1,p_2) = d_{\R\slash 2\pi\Z}(\theta_{p_1},\theta_{p_2}) := \min\{\abs{\theta_{p_1}-\theta_{p_2}+2\pi k} , k\in\Z \}.
$$
This means that the circle $\S^1$ is locally isometric to the real line $\R$. 


Unless specified, $\mu$ will denote a general probability measure on $(\S^1,\mathcal B(\S^1))$ where $\mathcal B(\S^1)$ is the Borel set of $\S^1\subset \R^2$. Given a point $p_0\in \S^1$, $\mu_{p_0}$ is the image measure of $\mu$ through  $e_{p_0}^{-1}:\S^1 \longrightarrow [-\pi,\pi[$. This is a measure on $\R$ with a support in $[-\pi,\pi[$ which is defined by 
\begin{equation}\label{eq.mes}
	\mu_{p_0}(A) = \mu \circ e_{p_0}(A \cap [-\pi,\pi[), \qquad \text{ for all } A \in\mathcal B(\R),
\end{equation}
where $\mathcal B(\R)$ is the Borel set in $\R$. The usual Euclidean mean/expectation and variance of $\mu_{p_0}$ are denoted 
$$
m(\mu_{p_0}) = \int_{\R} t d\mu_{p_0}(t). 
$$
Finally, let us introduce for any $p_0\in \S^1$, the map $F_{\mu_{p_0}}:[-\pi,\pi[\longrightarrow\R$ given by

\begin{align*}
	F_{\mu_{p_0}}(\theta) := F_\mu(e_{p_0}(\theta)) &= \frac{1}{2}\int_{-\pi}^\pi d_{\R\slash 2\pi\Z}^2(t,\theta) d\mu_{p_0}(t) \\ &= \frac{1}{2} \begin{cases}  {\int_{-\pi}^{\theta-\pi} (t +2\pi - \theta)^2d\mu_{p_0}(t) +  \int_{\theta-\pi}^{\pi} ( t-\theta )^2d\mu_{p_0}(t)},&
		\text{ if } 0\leq \theta< \pi, \\ 
		{\int_{-\pi}^{\theta+\pi} (t  -\theta)^2d\mu_{p_0}(t) +  \int_{\theta+\pi}^{\pi} ( t- 2\pi -\theta)^2d\mu_{p_0}(t)},  & 
		\text{ if } -\pi\leq \theta<0.
	\end{cases} 
	\end{align*}


\section{The Fréchet functional on the Circle} \label{part.ff}

\subsection{The derivative of the Fréchet functional}

A  function $f:[-\pi,\pi[\longrightarrow\R$ is said left continuous on $[-\pi,\pi[$ if it is left continuous everywhere on $]-\pi,\pi[$ and with $\lim_{\varepsilon\to0^-} f( \pi +\varepsilon) = f(-\pi)$. Similarly, $f$ is said to be continuous on $[-\pi,\pi[$ if it is left and right continuous on $[-\pi,\pi[$. We provide here an explicit expression of the derivative of $F_\mu$,

\begin{proposition}\label{prop,grad}
Let $\mu$ be a probability measure on $(\S^1,d_{\S^1})$ and fix an arbitrary $p_0\in\S^1$. Then, $F_{\mu}:\S^1 \longrightarrow\R$ is differentiable in following sense :
\begin{enumerate}
	\item Let $p\in\S^1$ be a point with a cut locus of $\mu$-measure 0, i.e $\mu(\{-p\}) =0$. Then $F_{\mu_{p_0}}$ is continuously differentiable at $\theta_p^{p_0}$ and we have
\begin{equation}\label{eq.grad}
\frac{d}{d\theta}F_{\mu_{p_0}}(\theta_p^{p_0}) = \begin{cases} {\displaystyle\theta_p^{p_0}  - 2\pi\mu_{p_0}([-\pi,-\pi + \theta_p^{p_0}[ ) - m( \mu_{p_0})}, &
\hfill\text{ if } 0\leq \theta_p^{p_0}<\pi, \\  {\displaystyle \theta_p^{p_0}  + 2\pi\mu_{p_0}([\pi+ \theta_p^{p_0},\pi [)  - m( \mu_{p_0})}, &
\text{ if } -\pi\leq \theta_p^{p_0}<0 .
\end{cases}
\end{equation}
\item The function $\frac{d}{d\theta}F_{\mu_{p_0}}$ is left continuous on $[-\pi,\pi[$. Then we extend the definition of the derivative of $F_\mu$ by  setting \emph{for all} $\theta\in]- \pi,\pi [$
\begin{equation}\label{eq.graddef}
\frac{d}{d\theta}F_{\mu_{p_0}}(\theta) := \lim_{\varepsilon \to 0^-} \frac{d}{d\theta}F_{\mu_{p_0}}(\theta +\varepsilon),
\end{equation}
and $\frac{d}{d\theta}F_{\mu_{p_0}}(-\pi) = \lim_{\varepsilon \to 0^-} \frac{d}{d\theta}F_{\mu_{p_0}}(\pi +\varepsilon)$.
\item Let $p\in\S^1$ be a point with a cut locus of positive measure, i.e $\mu(\{-p\}) >0$. Then, $p$ is a cusp point of $F_\mu$ in the sense that $ \lim_{\varepsilon \to 0^-} \frac{d}{d\theta}F_{\mu_{p_0}}(\theta_p^{p_0} +\varepsilon) -  \lim_{\varepsilon \to 0^+} \frac{d}{d\theta}F_{\mu_{p_0}}(\theta_p^{p_0} +\varepsilon) = -\mu(\{ -p\})$. \label{stat.3}
\end{enumerate}
\end{proposition}
\noindent Note that the left-continuity comes from our convention on the exponential map which is defined on $[-\pi,\pi[$. If a measure $\mu$ is such that  $\mu(\{ p \})=0$ for all $p\in\S^1$ then $F_\mu$ is of class $\mathcal C^1$ on $[-\pi,\pi[$. Differentiability issues appear when the measure $\mu$ has atoms, see Figure \ref{fig.1}. 

\begin{proof}
For convenience we omit in this proof the superscript $p_0$ by writing $\theta_p = \theta_p^{p_0}$ for all $p \in\S^1$. In the coordinate system centered at $p_0$ we have for all $\theta_p\in[-\pi,\pi[$
\begin{align}\label{eq.Ff}
F_{\mu_{p_0} }(\theta_p) =  \frac{1}{2}\int_{\R} t^2 d\mu_{p_0}(t ) - \theta_p m(\mu_{p_0})  &+ \frac{1}{2}\theta_p^2  + 2\pi\Big( g_{\mu_{p_0} }^+(\theta_p) \1_{[0,\pi[ } (\theta_p)+  g_{\mu_{p_0} }^-(\theta_p)\1_{[-\pi,0[ }, (\theta_p)\Big)
\end{align}
where 
$
g_{\mu_{p_0}}^+(\theta) =  \int_{-\pi}^{-\pi+\theta} (\pi + t - \theta )d\mu_{p_0}(t) \text{ and } g_{\mu_{p_0}}^-(\theta) = \int_{\theta+\pi}^{\pi} (\pi - t+\theta )d\mu_{p_0}(t).
$
Hence, to prove Proposition \ref{prop,grad}, we just have to study the derivative of $g_{\mu_{p_0}}^+$ and $g_{\mu_{p_0}}^-$. 

For all $\theta_p\in]0,\pi[$  and $\varepsilon\in\R$ such that $\theta_p + \varepsilon \in]0,\pi[$ we have,
 \begin{align}\label{eq.temp}
\frac{1}{\varepsilon} \Big(g_{\mu_{p_0}}^+(\theta_p+\varepsilon) - g_{\mu_{p_0}}^+(\theta_p)\Big) & = \frac{1}{\varepsilon} \int_{-\pi+\theta_p}^{-\pi+\theta_p+\varepsilon} (\pi + t - \theta_p )d\mu_{p_0}(t) -\int_{-\pi}^{-\pi+\theta_p+\varepsilon}  d\mu_{p_0}(t)
\end{align}
The limit from the left in equation \eqref{eq.temp} is $\lim_{\varepsilon\to 0^-} \frac{1}{\varepsilon} (g_{\mu_{p_0}}^+(\theta_p+\varepsilon) - g_{\mu_{p_0}}^+(\theta_p)) =  -\mu_{p_0}([-\pi,-\pi+\theta_{p}[)$ when $0<\theta_p<\pi $. 
The (left) derivative at $\theta_p=0$ is given by $\lim_{\varepsilon\to 0^-}\frac{1}{\varepsilon} \Big(g_{\mu_{p_0}}^-(\theta_p+\varepsilon) - g_{\mu_{p_0}}^+(\theta_p)\Big)=0$. 
Similarly, if $-\pi\leq \theta_p<0$, we have $\lim_{\varepsilon\to 0^-} \frac{1}{\varepsilon} (g_{\mu_{p_0}}^-(\theta_p+\varepsilon) - g_{\mu_{p_0}}^-(\theta_p)) =  \mu_{p_0}([\pi+\theta_{p},\pi[)$ and statement 2 is proved.


To prove statement $1$, suppose that the cut locus of $p$ is of $\mu$-measure 0. In this case, the limit from the left and from the right in equation \eqref{eq.temp} are equal as $\lim_{\varepsilon\to 0}\frac{1}{\varepsilon} \int_{-\pi+\theta_p}^{-\pi+\theta_p+\varepsilon} (\pi + t - \theta_p )d\mu_{p_0}(t) =0$ since $\mu_{p_0}(\{ \theta_{p} - \pi \}) = 0$. Thence, formula \eqref{eq.Ff} yields  $\frac{d}{d\theta}g_{\mu_{p_0}}^+(\theta_{p}) = -\mu_{p_0}([-\pi,-\pi+\theta_{p}[)$, if $\theta_{p}\in[0,\pi[  \text{ and } \frac{d}{d\theta}g_{\mu_{p_0}}^-(\theta_{p}) = \mu_{p_0}([\pi+\theta_{p},\pi[)$, if $  \theta_{p}\in[-\pi,0[.$


		Finally, suppose that $p\in\S^1$ has a cut locus of positive measure. If $0\leq \theta_p<\pi $, it means that $\mu_{p_0}(\{\theta_p - \pi\})>0$ and  
		then, $\frac{d}{d\theta} F_{\mu_{p_0}} (\theta_p) - \lim_{\varepsilon\to 0^+}  \frac{d}{d\theta} F_{\mu_{p_0}} (\theta_p + \varepsilon) = -\lim_{\varepsilon \to 0^+}\mu_{p_0}([-\pi + \theta_p,-\pi+\theta_{p}+\varepsilon  [) = -\mu_{p_0}(\{-\pi+\theta_{p} \})$. The case $-\pi\leq \theta_p<0$ is similar and the proof of statement 3 is completed.  
\end{proof}

\subsection{Local minimum of the Fr\'echet functional}

\begin{figure}[t]
\begin{center}
\includegraphics[width=5cm]{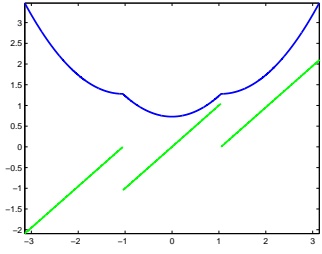}
\caption{Let $\mu= \tfrac{1}{6}\delta_{p_1} +\tfrac{2}{3}\delta_{p^*}  +\tfrac{1}{6}\delta_{p_2}$ with $p_1 = R_{\frac{2\pi}{3}}p^*$ and $p_2 =  R_{-\frac{2\pi}{3}}p^*$.  In blue: $F_{\mu_{p^*}}$. In green: $  \frac{d}{d\theta}F_{\mu_{p^*}}$.  }\label{fig.1}
\end{center}
\end{figure}

The critical points of the Fr\'echet functional are the points at which the derivative of $F_\mu$, in the sens of Proposition \ref{prop,grad}, is 0. 
As a immediate consequence of equation \eqref{eq.grad} we have 
\begin{equation}\label{eq:bar}
	\frac{d}{d\theta}F_{\mu_{p_0}}(\theta_p^{p_0}) = - m(\mu_{p} )
\end{equation}
Thus, the critical points are precisely the exponential barycenters (i.e. points $p\in\S^1$ at which $m(\mu_p) =0$). This fact was already shown in \cite{MR1187782} or \cite{MR1842295} for general manifolds. Note that a critical point of the Fr\'echet functional is not in general an extremum (see an illustration Figure \ref{fig.1}) but it is worth noticing that the minima of the Fr\'echet functional are regular points in the sens of the following result,

\begin{corollary}\label{prop.min}
Let $\mu$  be a probability measure on $\S^1$. The cut locus of a (local or global) minimum of $F_\mu$ is of $\mu$-measure 0. 
\end{corollary}

\begin{proof}
	Let $p \in \S^1$ be a point satisfying $\mu(\{-p\}) > 0$.  For any $p_0\in\S^1$, Statement \ref{stat.3} of Proposition \ref{prop,grad} ensure that the derivative $\frac{d}{d\theta} F_{\mu_{p_0}}$ of the Fr\'chet functional has a negative jump at $\theta^{p_0}_p$. Hence, the signs of
	$
	\left(\lim_{\varepsilon \to 0^-} \frac{d}{d\theta}F_{\mu_{p_0}}(\theta_{p_m}^{p_0} +\varepsilon),\lim_{\varepsilon \to 0^+} \frac{d}{d\theta}F_{\mu_{p_0}}(\theta_{p_m}^{p_0}+\varepsilon)\right)
	$ is either $(+,+)$, $(+,-)$ or $(-,-)$.  This means that $\theta_{p}^{p_0}$ cannot be a minimum of $F_{\mu_{p_0}}$ since it would correspond to the case $(-,+)$.
\end{proof}

\begin{remark}
	Note that assumptions of Corollary 1 in \cite{pen} and Theorem 1 in \cite{Bhatt} contain a condition of the form $\mu(\{\tilde p\})$ to ensure (classical) differentiability of the Fr\'echet functional at its minimum. In the case of the circle, Corollary \ref{prop.min} shows us that the Fr\'echet functional is \emph{automatically} differentiable at its minima.
\end{remark}

\section{Necessary and sufficient condition for the existence of the Fréchet mean} \label{part.th}

\subsection{Main result}


\begin{theorem}\label{th:crit}
Let $\mu$ be a probability measure and $p^*\in\S^1$ be a critical point of $F_\mu$. Then, the following propositions are equivalent,
\begin{enumerate}
\item $p^*$ is a well defined Fréchet mean of $(\S^1,\mu)$ \label{prop:1}.
\item For all $0<\theta<\pi$ \label{prop:3}
$$
\int_{0}^{\theta} \lambda([-\pi,-\pi+t[) -\mu_{p^*}([-\pi,-\pi+t[)dt >0,
$$ 
and  for all $-\pi\leq\theta<0$,
$$
\int_{\theta}^{0} \lambda(]\pi+t,\pi[) -\mu_{p^*}(]\pi+t,\pi[) dt >0 ,
$$
where $\lambda$ is the uniform measure on $[-\pi,\pi[$ and $\mu_{p^*}$ is defined in \eqref{eq.mes}.
\end{enumerate}
\end{theorem}

Theorem \ref{th:crit} gives a necessary and sufficient condition for the existence of the Fréchet mean of a general measure $\mu$ on the circle $\S^1$. This condition is given in terms of comparison between the $\mu$-measure and the uniform measure  $\lambda$ of balls centered at the cut locus of a global minimum. The important point is that the $\mu$-measure of a (small) neighborhood of the cut locus of $p^*$ cannot be larger than the uniform measure of this neighborhood.

As $\mu$ is a probability measure, the functions $t\longmapsto \lambda([-\pi,-\pi+t[) -\mu_{p^*}([-\pi,-\pi+t[) $ and  $t\longmapsto\lambda(]\pi-t,\pi[) -\mu_{p^*}(]\pi-t,\pi[)$ do not need to be  always positive for $t\in[-\pi,0[$ and $t\in[0,\pi[ $ respectively. %
An example where this quantity is always positive is when $\mu$ admits a density which is a decreasing function of the distance to a point $p^*$, see \cite{Kaziska20081314}. In this case, the density is radially distributed around its mode $p^*$ which is, by Theorem \ref{th:crit}, the Fr\'echet mean of $\mu$. Many classical probability distributions used in circular data analysis follow this framework: von Mises distribution, wrapped normal distribution, geodesic normal distribution \cite{pen} , etc...

Another well-known example of distributions that admit a well defined Fr\'echet mean is distributions with support included in an hemisphere, see \cite{Buss:2001:SAA:502122.502124}. More precisely, suppose that there exists a point $\hat p\in\S^1$ with  $\mu(\{p\in\S^1, d_{\S^1}(p,\hat p)\leq \frac{\pi}{2}\})=1 $ and $\mu(\{p\in\S^1, d_{\S^1}(p,\hat p)< \frac{\pi}{2}\})>0$. In this case, Statement \ref{prop:3} of Theorem \ref{th:crit} holds since one can show that the minimum $p^*$ of $F_\mu$ is in $ \{p\in\S^1, d_{\S^1}(p,\hat p)< \frac{\pi}{2}\}$ and that $F_{\mu_{p^*}}(\theta)-F_{\mu_{p^*}}(0) > \frac{1}{2} (\theta +\pi - 2\theta_{\hat p})^2$ for $\theta \in[-\pi,\theta_{\hat p} - \frac{\pi}{2}[$, $F_{\mu_{p^*}}(\theta)-F_{\mu_{p^*}}(0) = \frac{\theta^2}{2} $ for $\theta \in[\theta_{\hat p} - \frac{\pi}{2}, \theta_{\hat p}+\frac{\pi}{2}[$ and $F_{\mu_{p^*}}(\theta)-F_{\mu_{p^*}}(0) > \frac{1}{2} (\theta -\pi -2\theta_{\hat p} )^2$ for all $\theta \in[\theta_{\hat p}+\frac{\pi}{2},\pi[$. The case of equality corresponds to distributions with support in the boundary of the hemisphere, that is $\mu_{p^*} = (1 - \varepsilon) \delta_{\theta^{p^*}_{\hat p}-\frac{\pi}{2}} + \varepsilon \delta_{\theta^{p^*}_{\hat p}+\frac{\pi}{2}}$ with $\varepsilon = \frac{1}{\pi}\theta_{\hat p}^{p^*} + \frac{1}{2}$ and in this case, there are two global argmins at 0 and  $2\theta_{\hat p} \pm \pi$, see Figure \ref{fig.degener}.
\begin{figure}[h!]
\begin{center}
	\subfigure[]{\includegraphics[width=5cm]{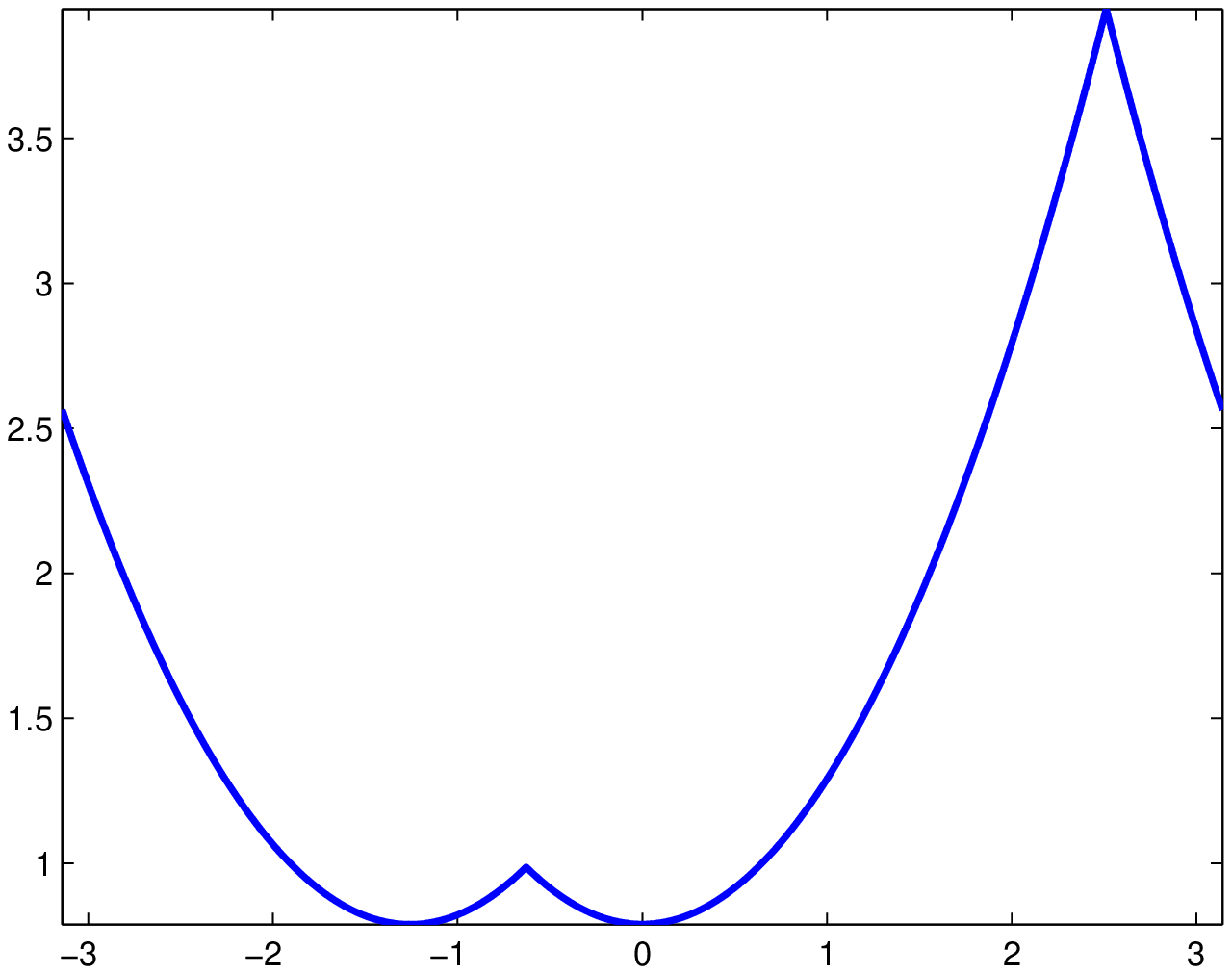}\label{fig.de1}}\qquad\qquad
	\subfigure[]{\includegraphics[width=5cm]{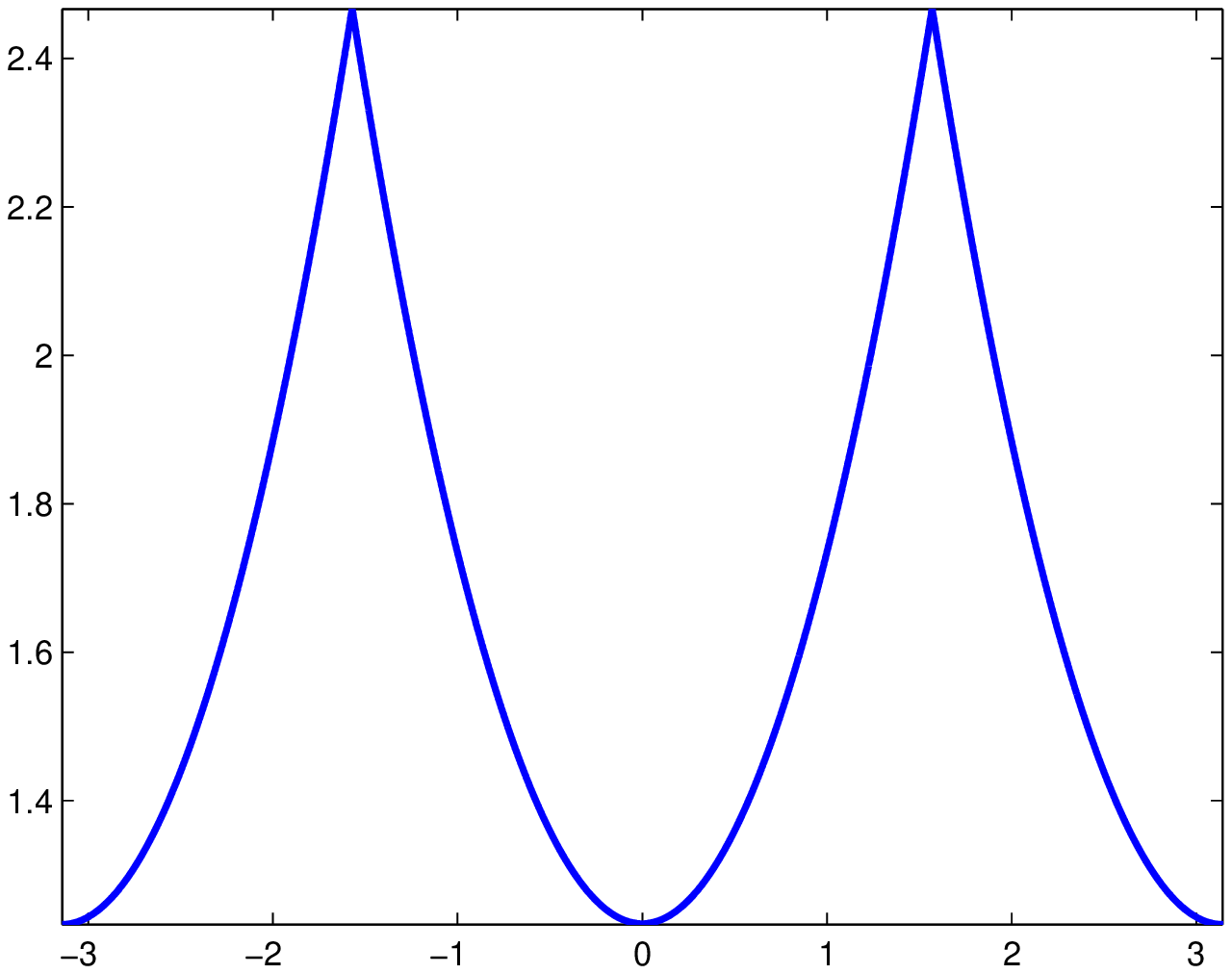}\label{fig.de2}}
	\caption{Let $\mu_\theta = (1 - \varepsilon) \delta_{\theta-\frac{\pi}{2}} + \varepsilon \delta_{\theta+\frac{\pi}{2}}$ with $\varepsilon = \frac{\theta}{\pi} + \frac{1}{2}$. \subref{fig.de1} $F_{\mu_\theta}$ with $\theta=0$. \subref{fig.de2}  $F_{\mu_\theta}$ with $\theta= \frac{3\pi}{10} $ }\label{fig.degener}
\end{center}
\end{figure}

\subsection{Proof of Theorem \ref{th:crit}}

As already noticed, there exists at least one global argmin $p^*\in\S^1$ of $F_\mu$ since $F_\mu$ is a continuous function defined on the compact set $\S^1$. Moreover, Proposition \ref{prop,grad} and Corollary \ref{prop.min} ensure that $p^*$ is a regular critical point of $F_\mu$, i.e. a zero of the derivative  with a cut locus of $\mu$-measure 0. 

In the normal coordinate system centered at $p^*$ the functional $G_{\mu_{p^*}}(\theta)= F_{\mu_{p^*}}(\theta) - F_{\mu_{p^*}}(0) $ has a particularly simple expression thanks to equation \eqref{eq:bar}. Indeed, we have
\begin{align*}
G_{\mu_{p^*}}(\theta) = \frac{1}{2}  \theta^2 + 2\pi \1_{[0,\pi[}(\theta) & \int_{-\pi}^{\theta-\pi} (\pi + t -\theta )d\mu_{p^*}(t)   + 2\pi  \1_{[-\pi,0[}(\theta)\int_{\theta+\pi}^{\pi} (\pi - t +\theta )d\mu_{p^*}(t).
\end{align*}
It is clear that Statement \ref{prop:1} is equivalent to the fact that the unique zero of $G_{\mu_{p^*}}$ is 0. Thence, Theorem 5.1 will be an easy consequence of Lemma \ref{lemme:g} below since we have $ \lambda([-\pi,-\pi+t[) =\frac{t}{2\pi}$ for all $0\leq t<\pi$.
\hfill $\Box $

\begin{lemma}\label{lemme:g} 
Let $\mu$ be a probability measure on $\S^1$ and $p^*\in\S^1$ be an argmin of $F_\mu$. Then for any $\theta\in[-\pi,\pi[$
$$
G_{\mu_{p^*}}(\theta) = 2\pi\begin{cases}{\displaystyle \int_{0}^{\theta} \frac{t}{2\pi} -\mu_{p^*}([-\pi,-\pi+t[)dt  },  & \text{ if } 0\leq \theta< \pi,   \\ {\displaystyle  \int_{\theta}^{0} \frac{-t}{2\pi} -\mu_{p^*}([\pi+t,\pi[) dt }, & \text{ if } -\pi\leq \theta<0 \end{cases}
$$
\end{lemma}

\begin{proof}
The probability measure $\mu$  can be decomposed as follow, 
\begin{equation}\label{eq:decomp}
\mu = a \mu^d + (1- a )\mu^\delta, \qquad  0\leq a \leq 1,
\end{equation}
where $\mu^d$ is a probability measure such that $\mu_d(\{p\}) =0$ for all $p\in\S^1$ and $\mu_\delta = \sum_{j=1}^{+\infty} \omega_j\delta_{p_j}$ where $\sum_{j=0}^{+\infty} \omega_j=1$ and the $p_j$'s are in $\S^1$. Hence, we consider the two cases separately : first, when the measure is non atomic, and then, when it is purely atomic. The general case follows immediately in view of equation \eqref{eq:decomp}. 

First, assume that $\mu$ is an atomless measure of $\S^1$. Proposition \ref{prop,grad} ensures that $F_{\mu}$ is continuously differentiable everywhere and the real function $F_{\mu_{p^*}}$ is of class $\mathcal C^1$ on $[-\pi,\pi[$. Formula \eqref{eq.grad} and the fundamental theorem of calculus gives for all  $\theta\in[-\pi,\pi[ $
\begin{equation}\label{eq.MA}
G_{\mu_{p^*}}(\theta) =  \int_{0}^{\theta}tdt - 2\pi\begin{cases} \int_{0}^{\theta} \mu_{p^*}([-\pi,-\pi+t[) dt, & \text{ if }\quad 0\leq \theta< \pi, \\ \int^{0}_{\theta} \mu_{p^*}([\pi+t,\pi[) dt, & \text{ if }\quad -\pi\leq \theta<0 .\end{cases}
\end{equation}
Consider now the case where $\mu$ is a purely atomic measure. First, we treat the case where the number of mass of Dirac in the sum is finite, i.e $\mu = \sum_{j=1}^n \omega_j \delta_{p_j}$, $n\in\NN$ .  Recall that $F_{\mu_{p^*}}$ is a Lipschitz function on $[-\pi,\pi[$. Proposition \ref{prop,grad} ensures that the derivative is piecewise continuous and formula \eqref{eq.grad} holds for all $\theta \in [-\pi,\pi[ \setminus \{\theta_{\tilde p_j}\}_{j=1}^n$, i.e points that have a cut locus of $\mu$-measure 0. 
Hence for all  $\theta\in[-\pi,\pi[ $, equation \eqref{eq.MA} holds too.

To treat the case where $ \mu = \sum_{j=1}^{+\infty} \omega_j\delta_{x_j} $ we proceed by approximation. Let $\phi(n) = \{ j\in\NN \ | \ \omega_j \geq \frac{1}{2^n} \}$ and remark that $\Card (\phi(n)) <+\infty$ for all $n\in\NN$ since $\sum_{j=1}^{+\infty} \omega_j =1$. Then if $\nu^n_{p^*} = \frac{1}{c(n)} \sum_{j\in\phi(n)} \omega_j \delta_{x_j}$, where $c(n) = \sum_{j\in\phi(n)} \omega_j $ is a normalizing constant, we have for all  $\theta\in[-\pi,\pi[ $,
$$
G_{\nu^n_{p^*}}(\theta) = \int_{0}^{\theta}tdt - 2\pi\begin{cases} \int_{0}^{\theta} \nu^n_{p^*}([-\pi,-\pi+t[) dt, & \text{ if }\qquad 0\leq \theta< \pi,\\ \int^{0}_{\theta} \nu^n_{p^*}([\pi+t,\pi[) dt, & \text{ if }\qquad -\pi\leq \theta<0. \end{cases}
$$
The sequence $(\nu^n_{p^*})_{n\geq 1}$ converges to $\mu$ in total variation. By the dominated convergence Theorem for all  $\theta_p\in [-\pi,\pi[$, $G_{\nu^n_{p^*}}(\theta)$ converge as $n\to\infty$ to \eqref{eq.MA}.
\end{proof}

\subsection{The criterion $P(\alpha,\varphi)$}

Although the necessary and sufficient condition of Theorem \ref{th:crit} is a key step to understand the problem of non uniqueness of the Fr\'echet mean on $\S^1$, it is of little practice interest: 
we have to know \emph{a priori} a critical point $p^*$ of the Fr\'echet functional. In this section, we derive sufficient conditions of existence with no restriction on the support of the probability measure and that are easily usable.
\begin{figure}[!t]
\begin{center}
\includegraphics[height=5cm]{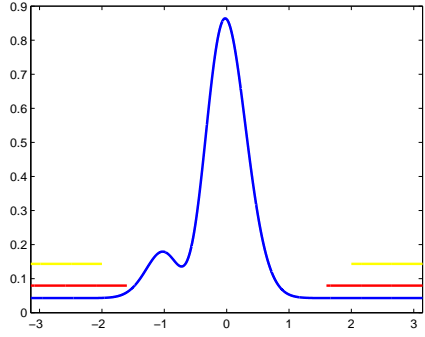} \label{CHAP3.fig.exemplecrit}
\caption[]{An example of distribution satisfying $P(\alpha,\varphi)$. 
Plot of the density $f_p$ in blue with the bounds $P(p,0.1,2)$ in yellow and $P(p,0.5,1.6)$ in red.
}\label{CHAP3.fig.exemplecritere}
\end{center}
\end{figure}
Let us introduce the following definition : 

\begin{definition}\label{def.P}
Let $f:\S^1 \longrightarrow \R^+$ be a probability density, $p\in\S^1$, $\alpha\in]0,1]$ and $\varphi\in]0,\pi[$. We say that $f$ satisfies the property $P(p,\alpha,\varphi)$ if for all $|\theta|\geq \varphi$
 \begin{equation}\label{eq.crit}
f_{p}(\theta) \leq \frac{1-\alpha}{2\pi},
\end{equation}
where $f_p = f\circ e_p : [-\pi,\pi[\longrightarrow\R^+$. Moreover, we say that $f \in P(\alpha,\varphi)$ if there is a $p\in\S^1$ such that $f$ satisfies  $P(p,\alpha,\varphi)$.
\end{definition}
\noindent The parameters $\alpha$ and $\varphi$ control the concentration of $\mu$ around $p$. The idea is to control the mass lying in the complementary of the ball $B(p,\varphi)$, 
see Figure \ref{CHAP3.fig.exemplecritere}. We have the following properties : 
\begin{lemma}\label{lemme.properties}
Let $f:\S^1\longrightarrow\R$ be a probability density on the circle. Then
\begin{enumerate}
\item $P(p,\alpha_1,\varphi_1) \Longrightarrow P(p,\alpha_2,\varphi_2)$ if $\alpha_1 \geq \alpha_2 $ and $\varphi_1 \leq \varphi_2$.
\item Let $p_1,p_2\in\S^1$ and $\varphi< \frac{\pi}{2}$. If $d_{\S^1}(p_1,p_2 )< \pi-\varphi$ then $P(p_1,\alpha,\varphi) \Longrightarrow P(p_2,\alpha, \varphi + d_{\S^1}(p_1,p_2 ) )$. 
\item  \label{lemm.prp3} If $f$ satisfies  $P(p,\alpha,\varphi)$ then $|m(\mu_p)| \leq \varphi + \frac{1-\alpha}{4\pi}(\pi -\varphi )^2 $.
\end{enumerate}
\end{lemma}
\begin{proof}
The first proposition is obvious in view of the Definition \ref{def.P}. 

To prove the second claim suppose that $0< \theta_{p_2}^{p_1} = -\theta_{p_1}^{p_2} \leq \pi-\varphi
$ (the other case is similar) and write 
$$
f_{p_2}(\theta) = f_{p_1}(\theta + \theta_{p_2}^{p_1})\1_{[-\pi,\pi-\theta^{p_1}_{p_2}[}(\theta) +  f_{p_1}(\theta + \theta_{p_2}^{p_1} -2\pi)\1_{[\pi-\theta^{p_1}_{p_2},\pi [}(\theta).
$$ 
In particular, it implies that $\pi>\pi -\theta_{p_2}^{p_1}\geq \varphi$ and since $P(p_1,\alpha,\varphi)$ holds, we have $f_{p_2}(\theta) \leq \frac{1-\alpha}{2\pi}$, if $\theta \geq \varphi - \theta_{p_2}^{p_1}$ or $\theta \leq -\varphi -\theta_{p_2}^{p_1}$. This is equivalent to the fact that $P(p_2,\alpha, \min\{ | -\varphi- \theta_{p_2}^{p_1}| ,| \varphi- \theta_{p_2}^{p_1}|  \}) =P(p_2,\alpha, \varphi + \theta_{p_2}^{p_1})$ holds. The case $ \varphi-\pi \leq \theta_{p_2}^{p_1} \leq 0$ is similar and we have $P(p_2,\alpha, \varphi - \theta_{p_2}^{p_1}) $. Finally recall that $|\theta_{p_2}^{p_1}| = d_{\S^1}(p_1,p_2 )$ and the property is proved.

To show the last claim, we only need to consider the case where $\mu_p$ has its support on $[0,\pi[$. Indeed, $\mu_p = \omega\mu_p^- + (1-\omega)\mu^+_p$ where $\mu_p^-([-\pi,0[) = \mu^+_p([0,\pi[) =1$ and $0\leq \omega\leq 1$. It yields that $m(\mu_p)= \int_{\R} t d(\omega\mu_p^- + (1-\omega)\mu^+_p) =\int_{\R^+} t d(- \omega\mu_p^- + (1-\omega)\mu^+_p) \leq \int_{\R^+} t d\mu^+_p$. 
Then, if the density $f_p$  of $\mu_p$ has its support in $[0,\pi[$ we have
\begin{align*}
m(\mu_p)& \leq \varphi\left(1-\int_\varphi^\pi f_p(t) dt \right)+ \int_\varphi^\pi t  f_p(t) dt 
\leq \varphi + \frac{1-\alpha}{2\pi}\int_\varphi^\pi (t-\varphi) dt,
\end{align*}
which gives the result.
\end{proof}

If the density $f$ is sufficiently concentrated around a critical point $p^*$ of $F_\mu$ then, this point is the Fréchet mean of $\mu$. More precisely we have the following result,

\begin{proposition}\label{prop:suf}
Let $\mu$ be a probability measure with density $f:\S^1\longrightarrow\R^+$ and $p^*\in\S^1$ be a critical point of $F_\mu$. 
If $f$ satisfies $P(p^*,\alpha,\varphi)$ with $\alpha\in]0,1]$ and $0<\varphi<\varphi_\alpha = \pi\frac{\sqrt{\alpha}}{1 + \sqrt{\alpha}}$ 
 then, $\mu$ admits a well defined Fréchet mean at $p^*$.
\end{proposition}
For all $\alpha\in]0,1]$, we have $\varphi_0 = 0 \leq \varphi_\alpha <  \frac{\pi}{2}=\varphi_1 $.  Note that if $\alpha=1$ then $\mu$ has its support included in the ball $B(p^*,\frac{\pi}{2}) =\{p\in\S,d_{\S^1}(p^*,p)\leq\tfrac{\pi}{2}\}$ and when $\alpha<1$, the support of $\mu$ can be the entire circle $\S^1$. 

\begin{proof}
Let $G_{\mu_{p^*}}(\theta) = F_{\mu_{p^*}}(\theta) - F_{\mu_{p^*}}(0)$ for all $\theta\in[-\pi,\pi[$. As the measure $\mu$ admits a density $f$, $G_{\mu_{p^*}}$ is twice differentiable and equation \eqref{eq.grad} implies that   
$\frac{d^2}{d\theta^2}G_{\mu_{p^*}}(\theta) =   1 - 2\pi f_{}(-\pi + \theta), \text{ if } 0\leq\theta<\pi$,  and  $\frac{d^2}{d\theta^2}G_{\mu_{p^*}}(\theta) =1 - 2\pi f_{}(\pi+\theta), \text{ if }-\pi\leq\theta<0$. Since $f\in P(p^*,\alpha,\varphi)$, the function $G_{\mu_{p^*}}$ is convex on $[-\pi+\varphi,\pi-\varphi]$ and has a unique minimum at 0. Let us show that 0 is the only argmin of $F_{\mu_{p^*}}$ on $[-\pi,\pi[$.  
If $\theta \in[\pi-\varphi,\pi[$, we have thanks to Lemma \ref{lemme:g}
\[
G_{\mu_{p^*}}(\theta) =  G_{\mu_{p^*}}(\pi-\varphi) + \int_{\pi-\varphi}^{\theta} t -2\pi\mu_{p^*}([-\pi,-\pi+t[)dt.
\]
Since $f\in P(p^*,\alpha,\varphi)$ we have $G_{\mu_{p^*}}(\pi-\varphi) \geq\frac{\alpha}{2} (\pi-\varphi)^{2}$ and the second term is bounded from  below by $\int_{\pi-\varphi}^{\pi} t -2\pi\nu([-\pi,-\pi+t[)dt$ where $\nu = \tfrac{1}{2}(\delta_{-\varphi} + \delta_{\varphi})$. It yields for $\theta \in[\pi-\varphi,\pi[$,
\begin{align}
G_{\mu_{p^*}}(\theta) 
&\geq\frac{1}{2} (  \alpha (\pi-\varphi)^{2} - \varphi^{2}).\label{ineq.1}
\end{align}
The right hand side of \eqref{ineq.1} is strictly positive if $\varphi<\varphi_{\alpha}= \pi\frac{\sqrt{\alpha}}{1 + \sqrt{\alpha}}$. Similarly, the same condition implies $G_{\mu_{p^*}}(\theta) > 0$ for $\theta\in[-\pi,-\pi+\varphi[$. 
\end{proof}

We are now able to define a functional class of densities that admit a well defined Fr\'echet mean \emph{without} restriction on the support of the measure. 

\begin{theorem}\label{theo.suf}
Let $0<\delta<\tfrac{1}{2}$ be a parameter of concentration and  $\mu$ a probability measure with density $f\in P(\alpha,\varphi)$  (see Definition \ref{def.P}) with 
$$
\alpha_\delta\leq\alpha\leq 1 \qquad \text{ and } \qquad \varphi \leq \delta\varphi_\alpha
$$ 
where $\alpha_\delta$ be the square of the root of $ (5 - 6\delta + \delta^2 )X^{3}+(1- \delta^2)X^{2}-(2\delta +1) X- 1$ that lies in $]0,1]$ and $\varphi_\alpha= \pi\frac{\sqrt{\alpha}}{1+\sqrt{\alpha}}$. Then $\mu$ admits a well defined Fr\'echet means.
\end{theorem}
Firstly, remark that there is no need to know a critical point \textit{a priori}. Secondly, the parameter $\delta$ controls the concentration of $f$ via the inequality $\alpha_\delta<\alpha$ and  $\varphi\leq \delta\varphi_\alpha$. There is a tradeoff between $\alpha$ and the possible value of $\varphi$: the smaller $\alpha$ is (i.e the less $f$ is concentrated) the smaller $\varphi$ must be (i.e we need to control the value of the density on a bigger interval). 
In Tabular  \ref{tab.1} we give examples of numerical values. Note that the column corresponding to $\delta =0 $  is given as a reference only as the set $P(\alpha_\delta,\delta\varphi_\alpha )$ is empty for this values of $\delta$.

\bigskip

\begin{table}[ht]
\begin{center} 
\begin{tabular}{cccccc}
\hline  $\delta =$ & 0 & $\frac{1}{10}$ & $\frac{1}{5}$ & $\frac{1}{3}$ & $\frac{1}{2}$ \\ 
\hline  $\alpha_\delta \leq$ & 0.39 & 0.46 & 0.54 & 0.69 & \phantom{0}1\phantom{0} \\ 
  $\delta\varphi_{\alpha_\delta }\geq$ & 0 &  0.12 & 0.26 & 0.47 & $\frac{\pi}{4}$ \\
\hline 
\end{tabular} 
\caption{Some values of $\alpha_\delta  $ and $\delta\varphi_{\alpha_\delta } $ depending on $\delta\in]0,\tfrac{1}{2}[ $.}\label{tab.1}
\end{center}
\end{table}
\begin{proof}
	We show that under the hypothesis of the Theorem \ref{theo.suf}, there is a critical point $p^*$ of $F_\mu$ satisfying $d_{\S^1}( p,p^*) \leq (1-\delta)\varphi_\alpha$ where $p\in\S^1$ is a point satisfiying $f\in P(p,\alpha,\varphi)$. Thence, by Lemma \ref{lemme.properties}, $f$ belongs to $P(p^*,\alpha, \delta\varphi_\alpha + (1-\delta)\varphi_\alpha) =P(p^*,\alpha, \varphi_\alpha)$ and  Proposition \ref{prop:suf} will ensure that $p^*$ is the Fréchet mean of $\mu$. 

	In the rest of the proof, we show that there is a $p^*\in\S^1$ such that $\frac{d}{d\theta}F_{\mu_p}(\theta^p_{p^*}) = 0$ with $d_{\S^1}( p,p^*) \leq (1-\delta)\varphi_\alpha$. To this end, suppose that $m(\mu_p) \geq 0$ (the case $m(\mu_p)<0$ is similar). Equation \eqref{eq:bar} implies that
$	\frac{d}{d\theta}F_{\mu_p} (0) = -m(\mu_p) \leq 0$, 
and let us check that, under the hypothesis of Theorem \ref{theo.suf}, we have $ \frac{d}{d\theta}F_{\mu_p}\left((1-\delta)\varphi_\alpha\right)\leq 0$. Since $\tfrac{d}{d\theta}F_{\mu_p}$ is a continuous function, the intermediate value Theorem will ensure the existence of a critical point $p^*$ such that $ |\theta_{p^*}^{p}|\leq (1-\delta)\varphi_{\alpha}$. Equation \eqref{eq.grad} gives  
$$
\frac{d}{d\theta}F_{\mu_p}\left((1-\delta)\varphi_\alpha\right) =(1-\delta)\varphi_\alpha - 2\pi \mu_{p}\left(\left[-\pi,-\pi + (1-\delta)\varphi_\alpha\right[\right) - m(\mu_p). 
$$
We have $- 2\pi \mu_{p}([-\pi,-\pi +(1-\delta)\varphi_\alpha[) \geq (\alpha -1 )(1-\delta)\varphi_\alpha$ since $f\in P(p,\alpha,\delta\varphi_\alpha)$  and $ \abs{-\pi +(1-\delta)\varphi_\alpha}\geq \delta\varphi_\alpha $. Moreover, $ - m(\mu)\geq  \delta\varphi_\alpha-\frac{1-\alpha}{4\pi}\left(\pi -\delta\varphi_\alpha \right)^2$ by Lemma \ref{lemme.properties} Statement \ref{lemm.prp3}. It gives,
\begin{align*}
\frac{d}{d\theta}F_{\mu_p}\left((1-\delta)\varphi_\alpha\right) & \geq 
\pi\frac{(5-6\delta + \delta^2)\alpha\sqrt{\alpha} + (1-\delta^2)\alpha-(2\delta +1)\sqrt{\alpha} -1 }{ 1+\sqrt {\alpha}}
\end{align*}
This quantity is positive as soon as $1\geq \alpha > \alpha_\delta$, where $\sqrt{\alpha_\delta}$ is  the root of the polynomial $X\longmapsto(5-6\delta + \delta^2)X^3 + (1-\delta^2)X^2-(2\delta +1)X -1$ that lies in $]0,1]$.  It is easy to see that $\alpha_{\frac{1}{2 }} = 1$ and numerical experiments show that the (increasing) function $\delta \longmapsto\alpha_\delta$ takes its value in $]0.39,1[$ for $\delta\in]0,\tfrac{1}{2}[$. 
\end{proof}

\section{Fréchet mean of an empirical measure} \label{part.emp}

\subsection{Existence}

Let $X_1,\ldots,X_n$ be independent and identically distributed random variables with value in $(\S^1,d_{\S^1})$ and of probability distribution $\mu$. The empirical measure is defined as usual by 
$$
\mu^n = \frac{1}{n} \sum_{i=1}^n \delta_{X_i}
$$
and we note $p^*_n$ the empirical Fréchet mean defined as the unique argmin of 
$
F_{\mu^n}(p)= \frac{1}{2n}\sum_{i=1}^n d^2_{\S^1}(p,X_i)$, $p\in\S^1.
$ 
In \cite{MR0501230} a strong law of large number is given for the empirical Fréchet mean in a semi metric space which is the case of $(\S^1,d_{\S^1})$. 
In particular, if $p^*_n$ exists for each $n\in\NN$, the empirical Fréchet mean is a consistent estimator of the Fréchet mean. Indeed, the empirical Fréchet mean is well defined almost surely for a wide class of probability measures as the following fact from \cite{Bhatt} Remark 2.6 shows,

\begin{lemma}\label{lemme.EFM}
Let $\mu$ be a non atomic  probability measure on the circle, i.e satisfying $\mu(\{p\}) =0$ for all $p\in\S^1$. Then for all $n\in\NN$ the empirical  Fréchet mean exists almost surely.
\end{lemma}

Hence, the empirical Fréchet mean $p_n^*$ of a probability measure $\mu$ can be computed even if $\mu$ does not possess a well defined Fréchet mean. 

\subsection{Consistency}

If the Fréchet mean $p^*$ of $\mu$ is well defined, we study the rate of convergence of the empirical Fréchet mean $p^*_n$ to $p^*$. 

\begin{proposition}\label{prop.ee}
Let $\mu$ be a measure with density $f:\S^1 \longrightarrow \R$ that admits a well defined Fréchet mean $p^*$.  Then, there exists a strictly increasing function $\rho : ]0,\pi[\longrightarrow ]0,+\infty[$ such that for all $p\in\S^1$, $p\neq p^*$ 
$$
F_{\mu}(p) - F_{\mu} (p^*) \geq \rho(d_{\S^1}(p,p^* )).
$$
If $p^*_n$ denotes the empirical Fréchet mean, we have  for all  
 $x>0$
 \begin{equation}\label{eq.ee}
	 \P\left ( \rho( d_{\S^1}(p^*_n, p^*) ) \geq C(s) \sqrt{\frac{x}{n}} \right ) \leq 2e^{-x}.
	 \end{equation}
where $s = \max\{ \abs{ x-y},\ x,y\in \supp(\mu)\}$ and $C(s)=(4\pi^2+ 4\pi^2 s+2s)\leq 4\pi(2\pi^2+\pi+1)$.
\end{proposition}
\noindent The function $\rho$ in the statement of the Proposition \ref{prop.ee} determines the rate of convergence of $p^*_n$ to $p^*$. 

\begin{proof}
The first claim about the lower bound $\rho$ is a direct consequence of the following Lemma:
\begin{lemma}\label{lem.rho}
	Let $f:[-\pi,\pi[ \longrightarrow \R^+$ be a continuous function on $[-\pi,\pi[$ (see Section \ref{part.ff})
	. If $f$ vanishes at a unique point $\theta_0\in[-\pi,\pi[ $, then there exists a strictly increasing function $\rho:]0,\pi[ \longrightarrow ]0,+\infty[$ such that for all $\theta\in[-\pi,\pi[\setminus\{0\}$,  $$f(\theta) \geq \rho(d_{\S^1}(\theta_0, \theta)).$$
\end{lemma}

\begin{proof}
As $f$ is the restriction on $[-\pi,\pi[$ of a continuous periodic function we can assume, without loss of generality, that $\theta_0 = 0$. Then, define for all $\theta\in ]0,\pi[$
\begin{equation*}
	\rho (\theta)= \frac{1}{|\theta|}\int_{0}^{\theta} g(t)dt, \qquad \text{ where } g(t) = \min\left\{ f(-t) ,f(t), \min_{\abs{\tau}>t} f(\tau)\right\}, \text{ for all } t\in[-\pi,\pi[.\tag*{\qedhere}
\end{equation*}
\end{proof}

We now focus on the proof of the concentration inequality \eqref{eq.ee} which is divided in two steps: first we show the uniform convergence in probability of $F_{\mu^n_{p}}$ to $ F_{\mu_{p}}$ and then, we deduce the convergence of their argmins by using the lower bound given by the function $\rho$. 
Equation \eqref{eq.grad} implies that 
\begin{align}
2\sup_{\theta\in[-\pi,\pi[} \babs{F_{\mu_{p}^n}(\theta) - F_{\mu_{p}}(\theta)} & \leq 2\pi \sup_{\theta\in[-\pi,\pi[} \babs{\frac{d}{d\theta}F_{\mu_{p}^n}(\theta) -\frac{d}{d\theta} F_{\mu_{p}}(\theta)} + 2\babs{F_{\mu_{p}^n}(0) - F_{\mu_{p}}(0)} \nonumber \\
&\leq  4\pi^2 \abs{ m(\mu_p) - m(\mu_{p}^n)}+ 2 \abs{m_2 (\mu_{p})- m_2(\mu_p^n)}  \label{eq.ub1}\\ & \qquad \qquad + 4\pi^2 \sup_{\theta\in[-\pi,\pi[} \abs{\mu_{p}([-\pi,\theta[) -\mu^n_{p}([-\pi,\theta[)} \label{eq.min1} 
\end{align}
where $m_2(\nu) = \int_\R t^2 d\nu(t)$, for a measure $\nu$ on $\R$. The term \eqref{eq.min1} can be controlled in probability using the Dvoretzky-Kiefer-Wolfowitz inequality (see e.g \cite{MR1062069}), and we have for all $x>0$, 
$
\P\Big(4\pi^2 \sup_{\theta\in[-\pi,\pi[} \abs{\mu_{p}([-\pi,\theta[) -\mu^n_p([-\pi,\theta[)} \geq 4\pi^2\sqrt{\frac{x}{n}}  \Big) \leq 2e^{-x}.
$
For the terms of \eqref{eq.ub1} which involve the first and second moment of $\mu_p$ and  $\mu_p^n$, we use an  Hoeffding type inequality which gives for all $x>0$, 
$
\P\Big ( 4\pi^2\abs{m(\mu_{p})  -  m(\mu_p^n)} + 2\abs{m_2(\mu_{p})  -  m_2(\mu_p^n)} \geq s (4\pi^2 + 2s  )\sqrt{\frac{x}{n}} \Big ) \leq 2e^{-x},
$
where $s=\max\{ \abs{ x-y},\ x,y\in \text{support } \mu\}$ is the diameter of the support of $\mu$. Combining these two concentration inequalities 
we have for all $x>0$,
\begin{equation}\label{eq.rho}
\P\left (2\sup_{\theta\in\R}\abs{F_{\mu_{p}}(\theta) -F_{\mu_{p}^n} (\theta) } \geq (4\pi^2+ 4\pi^2 s+2s)\sqrt{\frac{x}{n}}\right ) \leq 2e^{-x}.
\end{equation}

We now use a classical inequality in M-estimation,
$
\big|F_{\mu_{p}}(\theta_{p^*_n}) - F_{\mu_{p}}(\theta_{p^*})\big| 
\leq 2\sup_{\theta\in[-\pi,\pi[} \big|F_{\mu_{p}^n}(\theta)  - F_{\mu_{p}}(\theta)\big|.
$
By Lemma \ref{lem.rho}, there exists an increasing function $ \rho:\R^+ \longrightarrow\R^+$ such that 
$
F_{\mu_{p}}(\theta_{p^*_n}) - F_{\mu_{p}}(\theta_{p^*})\geq \rho(d_{\S^1}(\theta_{p^*_n} , \theta_{p^*} ) ).
$
%
%
%
Plugging this in equation \eqref{eq.rho} we have,
$$
\P\left (\rho(d_{\S^1}(\theta_{p^*_n} ,\theta_{p^*}) ) \geq(4\pi^2+ 4\pi^2 s+2s)\sqrt{\frac{x}{n}}\right ) \leq 2e^{-x},
$$
and the proof of Proposition \ref{prop.ee} is completed.
\end{proof}

The function $\rho$ that appears in the statement of Proposition \ref{prop.ee} can be explicitly computed if 
the  density $f\in P(\alpha,\varphi)$. 
The parameter $\alpha\in]0,1]$ can be interpreted as a measure of the convexity of $F_{\mu_{p}}$ on the interval $[-\varphi, \varphi]$. For example, if $\alpha =1$ and $\varphi = \varphi_\alpha=\frac{\pi}{2}$, then $\mu$ has its support contained in $[-\frac{\pi}{2},\frac{\pi}{2}]$ and $F_{\mu_{p}}$ is quadratic on $[-\frac{\pi}{2},\frac{\pi}{2}]$. 

\begin{proposition} 
Let $\mu$ be a probability measure with density $f:\S^1\longrightarrow\R^+$ satisfying the hypothesis of Theorem \ref{theo.suf}. Note $p^*$ the Fr\'echet mean of $\mu$ and for all $x>0$  we have
$$
\P\left ( d_{\S^1}(p^*_n,p^*) \geq \sqrt{ B(\alpha,\varphi)  } \left(\frac{x}{n}\right)^{\tfrac{1}{4}} \right ) \leq 2e^{-x},
$$
where $B(\alpha,\varphi) =  C \max \left\{ \frac{\pi^2}{\gamma(\alpha,\varphi)}, \frac{2}{\alpha}\right\}  $ with $\gamma(\alpha,\varphi) = \frac{1}{2} (  \alpha(\pi-\varphi)^{2}-\varphi^{2})$ and $C = 4\pi(2\pi^2+\pi+1)$.
\end{proposition}

\begin{proof}
This result is a direct consequence of Proposition \ref{prop.ee} and we only have to find a strictly increasing function $\rho:[0,\pi]\longrightarrow\R^+$ satisfying for all $\theta\in [-\pi, \pi[$, 
$
F_{\mu}(p) - F_{\mu} (p^*) \geq \rho(d_{\S^1} (p,p^* )).
$
As $\mu_{p^*}$ admits a density $f_{p^*}$, 
the Fréchet functional $F_{\mu_{p^*}} $ is twice differentiable and $f \in P(p^*,\alpha,\varphi_\alpha)$, see proof of Theorem \ref{theo.suf}. Thus, for all $\theta \in [-\pi+\varphi_\alpha,\varphi_\alpha-\pi]$, 
a second order Taylor expansion of $F_{\mu_{p^*}}$ at $0$  ensures that for some $\tilde \theta\in [-\pi+\varphi_\alpha,\pi-\varphi_\alpha]$,
$$
F_{\mu_{p^*}}(\theta) - F_{\mu_{p^*}} (0) = \frac{1}{2} \theta^2 \frac{d^2}{d\theta^2}F_{\mu_{p^*}} (\tilde\theta)\geq\frac{\alpha}{2}\theta^2 .
$$
For all $\theta \in [-\pi,-\varphi_\alpha[ \cup  ]\varphi_\alpha,\pi[$ we have by inequality \eqref{ineq.1}, 
$$
F_{\mu_{p^*}}(\theta) - F_{\mu_{p^*}} (0) \geq \frac{1}{2} (  \alpha(\pi-\varphi_{\alpha})^{2}-\varphi_{\alpha}^{2}) =\gamma(\alpha,\varphi_{\alpha})>0.
$$ 
Then, let $\rho ( t)=  t^2 \min\{ \frac{\gamma(\alpha,\varphi_\alpha)}{\pi^2}, \frac{\alpha}{2}\} \geq t^2 \min\{ \frac{\gamma(\alpha,\varphi)}{\pi^2}, \frac{\alpha}{2}\} $ for any $\varphi \leq \varphi_\alpha$. 
%
%
\end{proof}
 \subsection{Computation}

Computation of the Fréchet mean of a general probability measure may not be an easy task as it is a \emph{global} optimization problem. In practice the Fréchet functional is not a convex function and a gradient descent algorithm will only give a \emph{local} minimum which depends on the initialization point. 

 \begin{figure}[h!]
\begin{center}
	\subfigure[]{\includegraphics[width=5cm]{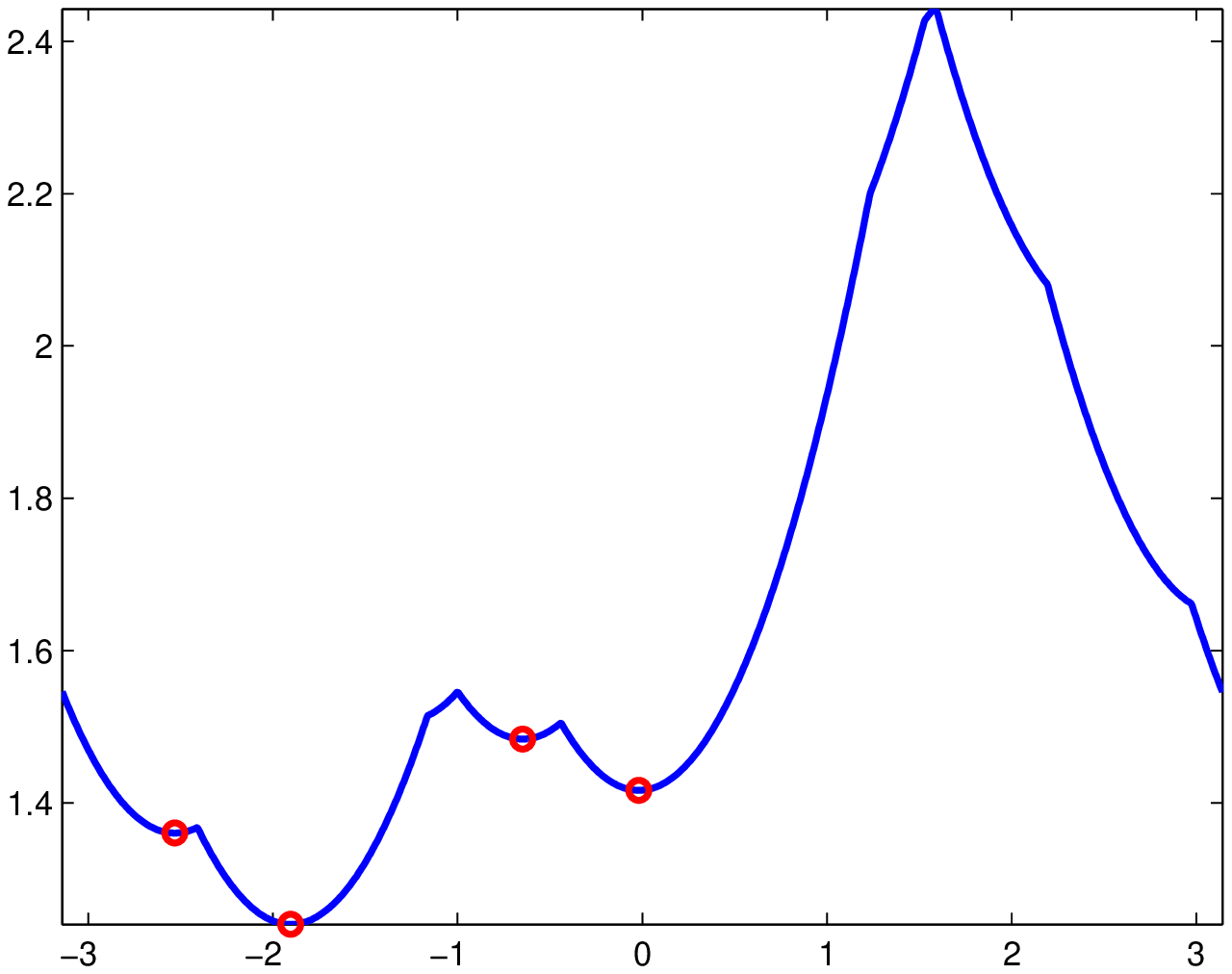}\label{fig.ex1}} 
	\subfigure[]{\includegraphics[width=5cm]{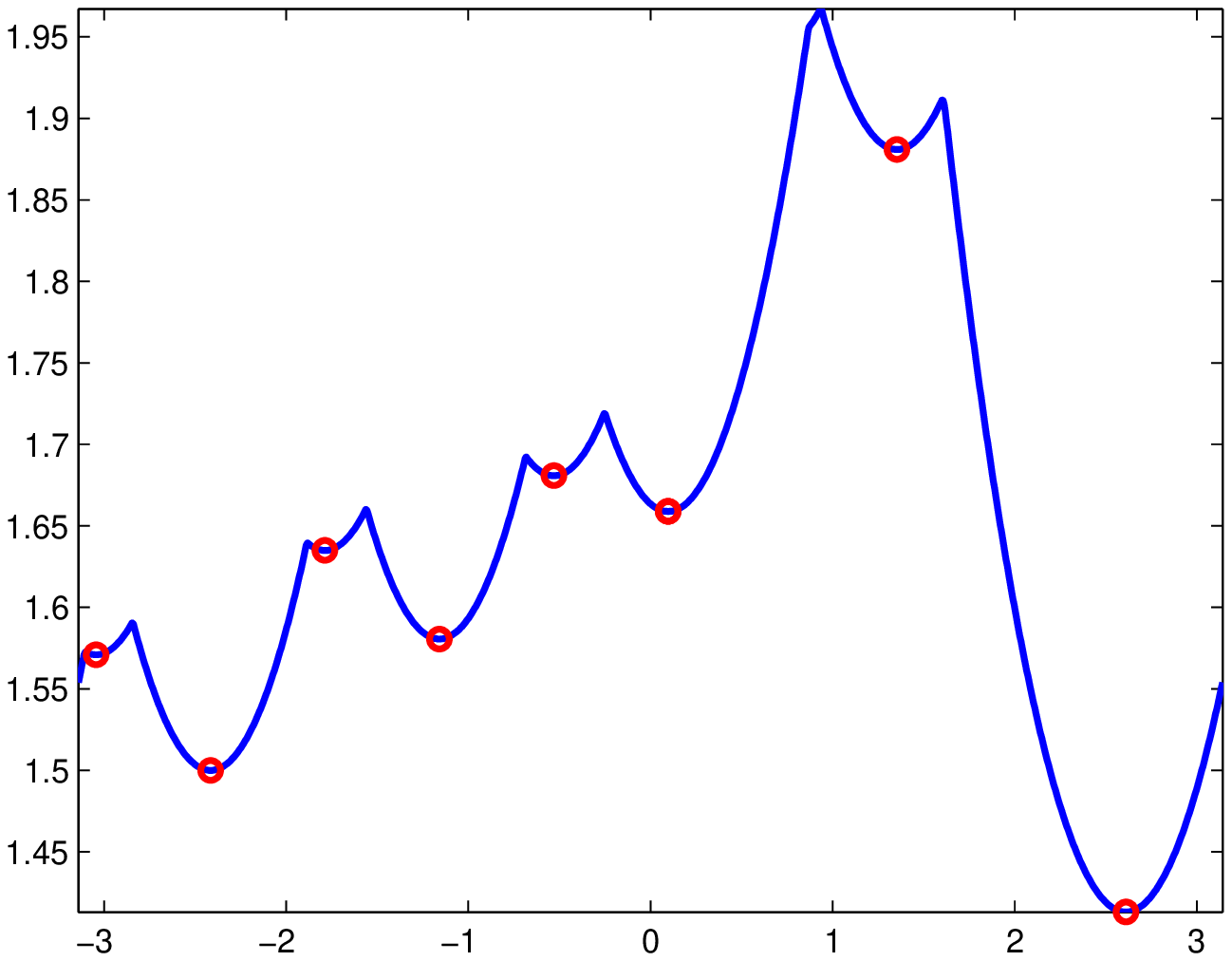}\label{fig.ex2}}
	\subfigure[]{\includegraphics[width=5cm]{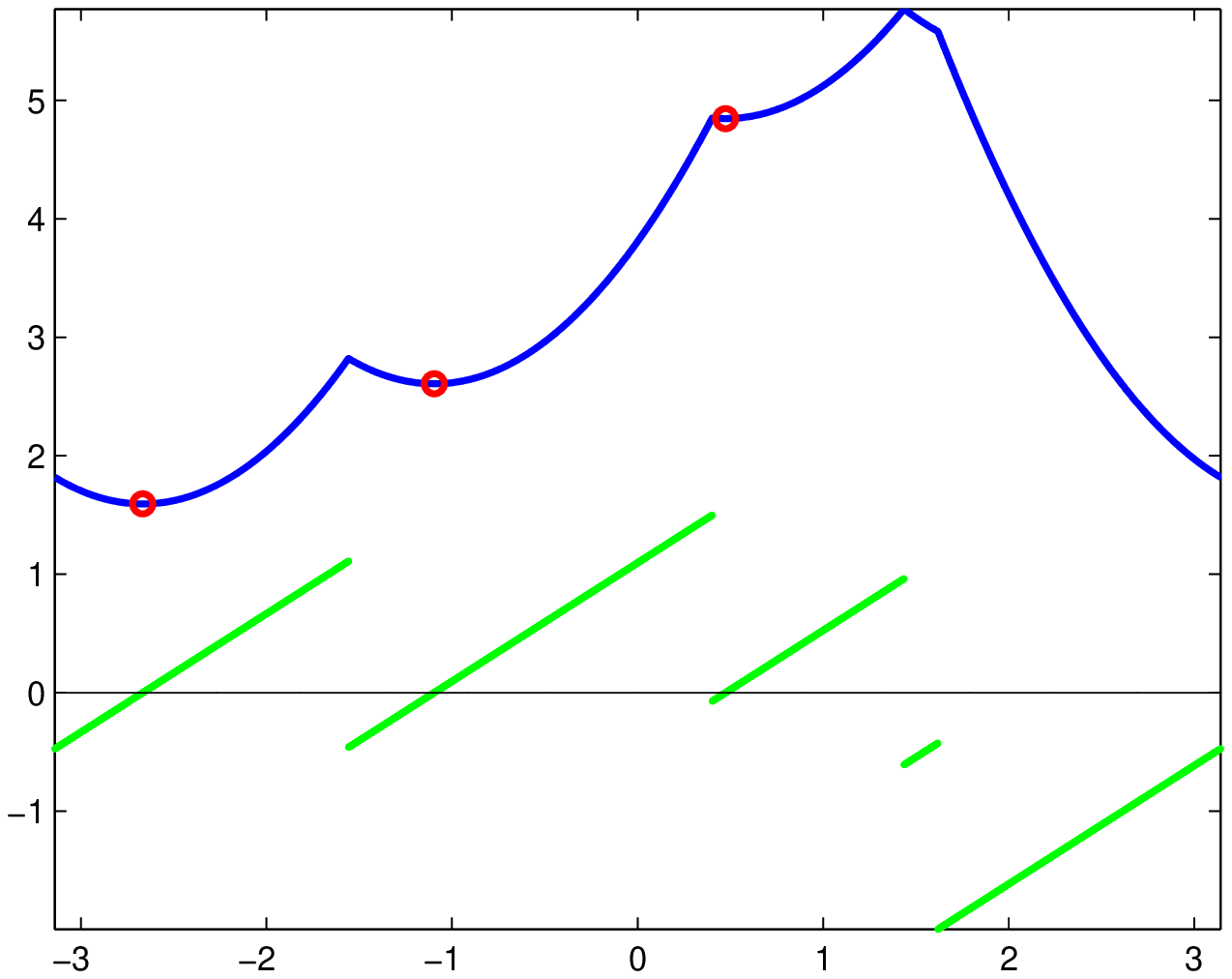}\label{fig.ex3}}
	\caption{\subref{fig.ex1} and \subref{fig.ex2} Plots of $F_{\mu^n}$ where $n=10$ and $\mu^n= \sum_{i=1}^n \delta_{X_i}$ where the $X_i$ i.i.d of uniform law $\lambda$. The red points are the local minima computed with the described algorithm. \subref{fig.ex3} In blue: plot of $F_{\mu^n}$ where $n=4$. In green: the derivative of $F_{\mu^n}$. The critical points are given by the intersection between the green curve and the $x$-axis in black.}\label{fig.2}
\end{center}
\end{figure}

The Fr\'echet functional of an empirical measure is a continuous piecewise quadratic function as it can be written $F_{\mu^n_p}(\theta) = \frac{1}{2n} \sum_{i=1}^n \big(\theta - \theta^p_{X_i} + 2\pi (\1_{[\pi+\theta,\pi[}(\theta_{X_i}^p)\1_{[-\pi,0[}(\theta)-\1_{[-\pi,\theta-\pi[}(\theta_{X_i}^p)\1_{[0,\pi[}(\theta)) \big)^2  $. This formula together with Corollary \ref{prop.min} implies that the regular critical points (i.e point with no mass at their cut locus) of $F_\mu$ are precisely the local minima of $F_\mu$. Moreover, the cumulative distribution function of $\mu^n$ is, here, piecewise constant with jumps of size $\frac{1}{n}$ and we have
$$
\mu^{n}_p ([-\pi, t[) = \frac{1}{n}\sum_{i=1}^n \1_{[-\pi,t[}(\theta_{X_i}^p) = \frac{1}{n}\Card \{\theta_{X_i}^p < t \}. 
$$
Thence, the derivative of the empirical Fr\'echet functional is piecewise linear and to find the critical points amounts to solve $n$ affine relations given by equation \eqref{eq.grad}, see also Figure \ref{fig.ex3} for an illustration. Note, that in practice, there are less than $n$ solutions, see  Figure \ref{fig.ex1} and \ref{fig.ex2}. 

\bigskip

The following algorithm takes as input the values $\{X_i\}_{i=1}^n$ and returns the Fréchet mean of $\mu^n = \frac{1}{n}\sum_{i=1}^n \delta_{X_i}$. 

\bigskip

\noindent \textbf{Initialization Step:} Choose an arbitrarily point $p\in\S^1$. \\ 
 Compute the coordinates $\{\theta_{X_i}^p\}_{i=1}^n$ and reorder them is increasing order. We denote $\tau^-_0=-\pi\leq\tau^-_1 \leq \tau^-_2 \leq \ldots\leq \tau^-_{n_1} <0 =\tau^-_{n_1 +1} $ the $n_1$ negative sorted terms and $\tau^+_0=\pi> \tau^+_1 \geq \ldots \geq \tau^+_{n_2}\geq 0=\tau^+_{n_2 +1}$ the  $n_2 = n - n_1$ positive sorted terms.\\
Compute the mean $m(\mu_p^n) = \frac{1}{n}(\tau^-_1 + \ldots + \tau^-_{n_1} +\tau^+_1 + \ldots +  \tau^+_{n_2})$ and initialize $\theta_{p^*}^p$ to $0$, says.\\

\noindent \texttt{\# The first step compares all the local minima in $[0,\pi[$ } 

\noindent\textbf{Step 1:} For i from $0$ to $n_1$ do \\\texttt{\# $\theta^p_{p^*,new}$  is the candidate to be a critical point between $\tau^-_i $ and $\tau^-_{i+1} $} \\
$\qquad$  Let $\theta^p_{p^*,new}=2\pi\frac{i}{n} + m(\mu_p^n)$\\ \texttt{\# verify if $\theta^p_{p^*,new}$  is a critical point and then test its value. If better, keep it.} \\
$\qquad$    if  $\tau^-_i +\pi\leq \theta^p_{p^*,new} \leq \tau^-_{i+1}+\pi$ and $F_{\mu^n_p}( \theta^p_{p^*,new}) \leq F_{\mu^n_p}( \theta_{p^*}^{p} )  $ then $\theta^p_{p^*} := \theta^p_{p^*,new}$ end if\\ 
end for.

\noindent \texttt{\# The Step 2 is the same as Step 1 but for local minima in $[-\pi,0[$ } 

\noindent\textbf{Step 2:}   For i = $0$ to $n_2$ do \\
$\qquad$  Let $\theta^p_{p^*,new}= - 2\pi\frac{i}{n} + m(\mu_p^n)$\\ 
$\qquad$    if  $\tau^+_{i+1} - \pi \leq \theta^p_{p^*,new} \leq \tau^+_{i}- \pi $ and $F_{\mu^n_p}( \theta^p_{p^*,new}) \leq F_{\mu^n_p}( \theta_{p^*}^{p} )  $ then $\theta^p_{p^*} := \theta^p_{p^*,new}$ end if\\
end for.
%
%
%

\noindent \texttt{\# The value of  $\theta^p_{p^*}$ is the best argmin } 

\noindent \noindent\textbf{Output:} Return $p^* = e_p(\theta^p_{p^*})$. 

\bigskip

This algorithm can be extended to more general measures that the empirical one. The approach will be the same: find the critical points of the Fréchet functional with formula  \ref{eq.grad} and compare the values of the local minima. Unfortunately, there may be some computational issues as  general cumulative distribution function will be not piecewise constant anymore.

\section{Conclusion}

It is not straightforward to extend criterion such as the one given in Theorem \ref{th:crit} to more general spaces, e.g. for the $n$ dimensional sphere $\S^n$. Recall that the circle $\S^1$ is a flat space in the sense that it is locally isometric to the Euclidean space $\R$. Then, the only phenomenon that induces uniqueness issues of the Fréchet mean is the presence of a cut locus.  The criterion presented in this note relies on an explicit formula for the gradient of the Fréchet mean. Curvature has an extra effect on the metric and makes difficult to derive exact computation on the Fréchet  functional and its gradient. Moreover, it is not clear if the role played by the uniform measure as a benchmark in the well definiteness of the Fréchet mean in $\S^1$ can be extended to $n$-spheres or non flat manifolds. 

\bigskip

\noindent{\bf Acknowledgements.}   The author would like to thank Dominique Bakry for help and encouragement during the writing of this paper and Jérémie Bigot for his careful reading of the manuscript. The author acknowledges the support of the French Agence Nationale de la
Recherche (ANR) under reference ANR-JCJC-SIMI1 DEMOS.

\bibliographystyle{plain}
\bibliography{biblio}
\end{document}